\documentclass[draft,reqno,10pt]{amsart}
 \usepackage{amssymb,amsmath}
   \usepackage{geometry, color}

\numberwithin{equation}{section}

\allowdisplaybreaks

\usepackage{color}

\newtheorem{theorem}{Theorem}[section]
\newtheorem{proposition}[theorem]{Proposition}

\newtheorem{lemma}[theorem]{Lemma}

\theoremstyle{definition}

\theoremstyle{remark}

\renewcommand{\epsilon}{\varepsilon }

\newcommand{\R}{\mathbb{R}}

\newcommand{\ds}{\displaystyle}

\renewcommand{\)}{\right)}

\newcommand{\A}{\mathcal A}

\begin{document}
\title[Infinitely many non-radial solutions to a critical equation on annulus]%
{Infinitely many non-radial solutions to a critical equation on annulus
 }

\author{  Yuxia Guo, \ \ Benniao Li, \ \   Angela Pistoia \ \ and \ \
Shusen Yan}

\address{  Department of Mathematical Science, Tsinghua University, Beijing, P.R.China}
\email{yguo@math.tsinghua.edu.cn}

\address{Department of Mathematics, The University of New England, Armidale, NSW 2351, Australia}
\email{bli9@une.edu.au}

\address{Dipartimento SBAI,  Via Scarpa 16, Roma, Italy }
\email{angela.pistoia@uniroma1.it}

\address{Department of Mathematics, The University of New England, Armidale, NSW 2351, Australia}
\email{syan@turing.une.edu.au}

\subjclass{Primary 35J91; Secondary 35B33}

\keywords {  Critical exponent, Non-radial solutions, Annulus domain}

\date{}


\begin{abstract}
In this paper, we build infinitely many non-radial sign-changing solutions to the critical problem:
\begin{equation*}
\left\{\begin{array}{rlll}
-\Delta u&=|u|^{\frac{4}{N-2}}u, &\hbox{ in }\Omega,\\
u&=0,  &\hbox{ on }\partial\Omega.
\end{array}\right. \eqno(P)
\end{equation*}
on the annulus $\Omega:=\{x\in \mathbb{R}^N: a<|x|<b\}$, $N\geq 3.$
In particular, for any integer $k$ large enough, we build a non-radial  solution  which look  like the unique positive solution $u_0$  to $(P)$ crowned by $k$ negative bubbles arranged on a regular polygon with radius $r_0$ such that  $r_0^{\frac{N-2}{2}}u_0(r_0)=:\displaystyle\max_{a\leq r\leq b}r^{\frac{N-2}{2}}u_0(r).$

\end{abstract}

\maketitle


\section{Introduction}

This paper deals with the existence of solutions to the critical  elliptic problem:
\begin{equation}\label{main}
\left\{\begin{array}{rlll}
-\Delta u&=|u|^{4\over N-2'}u, &\hbox{ in }\Omega,\\
u&=0,  &\hbox{ on }\partial\Omega,
\end{array}\right.
\end{equation}
where $\Omega$ is a bounded domain in $\mathbb R^N$ and $N\ge3.$
\vskip8pt

It is well known that the geometry of the domain $\Omega$ plays a crucial role in the solvability of the problem \eqref{main}.
Indeed, if $\Omega$ is a star-shaped domain, the classical  Pohozaev identity \cite{poho} implies that \eqref{main} does not have any solutions. While if $\Omega=\{x\in\mathbb R^N\ :\ a <|x|<b\}$ is an annulus,
Kazdan and Warner \cite{kw} found a positive solution and  infinitely many radial sign-changing solutions. Without any symmetry assumptions, the existence of solutions is a delicate issue. The first existence result is due to Coron in \cite{co} who proved that problem \eqref{main} has  a positive solution in domain $\Omega$ with a small hole. Later, Bahri and Coron in \cite{bc} proved that actually a positive solution alwasys exists as lonf as the domain
 has non-trivial homology with $\mathbb Z_2$-coefficients. However, this last condition is not necessary since   solutions to problem \eqref{main} in contractible domains have been found by   Dancer \cite{da}, Ding \cite{di}, Passaseo \cite{pa1,pa2} and
Clapp and Weth \cite{cw}.
The existence of sign-changing solutions is an even  more delicate issue and it is  known only for domains which have some symmetries or a small hole. The first existence result is due to Marchi and Pacella \cite{mapa} for symmetric domains with thin channels. Successively, Clapp and Weth \cite{cw} found  sign-changing solutions in a symmetric domain with a small hole. A first attempt to remove the symmetry assumption is due  to   Clapp and Weth \cite{cw2}, who  found a second solution to \eqref{main}  in a domain with a small hole but they were not able to say if it changes sign or not. Sign-changing solutions in a domain with a small hole have been found by Clapp, Musso and Pistoia in \cite{cmp}.
Recently, Musso and Pistoia    \cite{mp} and Ge, Musso and Pistoia \cite{gmp} (see also \cite{gmpd}) proved that in a domain (not necessarily symmetric) with a small hole  the number of sign-changing solutions to problem \eqref{main} becomes arbitrary large as the size of the hole decreases.
The existence of a large number of sign-changing solutions in a domain   with a hole of arbitrary size  is due to Clapp and Pacella in \cite{cp}, provided the domain has enough symmetry.

  \bigskip

It is largely open  for the problem of the existence of infinitely many sign-changing solutions in a general domain with non-trivial homology in the spirit of the famous Bahri and Coron's result.
  \bigskip

Here, we will focus on the existence of infinitely many sign-changing solutions to problem \eqref{main} when  $\Omega:=\{x\in \mathbb{R}^N: a<|x|<b\}$ is an annulus.
  The existence of infinitely many radial solutions   was established
by Kazdan and Warner in \cite{kw}. On the other hand,
  an annulus is invariant under many group actions and then it is natural to expect non-radial solutions which are invariant under these group actions.
 Indeed, Y.Y. Li in \cite{liyy} improved a previous result by Coffman \cite{c} and  he found for any integer $k\ge1$  in a sufficiently thin annulus  some non-radial solutions which are invariant under the action of the group
$\mathfrak G_k\times\mathfrak O(N-2),$ when $N\ge4.$  Here $\mathfrak O(N-2)$ denotes the group of orthogonal $(N-2)\times (N-2)$ matrices and    $\mathfrak G_k $ is the subgroup of  matrices which rotates $\mathbb R^2$
with angles equal to integer multiple of $2\pi\over k$.
Recently, Clapp in \cite{cl} found infinitely many non-radial solutions    which are invariant under the action of a suitable group
 whose  orbits are infinite, provided $N=4$ or $N\ge6.$

  \bigskip
  
 In this paper we prove the existence of infinitely many new non-radial solutions which  are invariant under the action of a group whose orbits are finite and they are not invariant
 under the action of the group $\mathfrak G_k\times\mathfrak O(N-2).$ Moreover, as far as we know, this is the first example of non-radial solutions in the $3-$dimensional annulus.
 \\

 Let us state our main result. Let $\Omega:=\{x\in \mathbb{R}^N: a<|x|<b\}$ be an annulus. Assume that
 \begin{equation}\label{nondegene}
 \hbox{ \it{the unique positive radial solution  $u_0$ to \eqref{main} is non-degenerate}.}
 \end{equation}
The uniqueness has been proved by Ni and Nussbaum \cite{ni}. The non-degeneracy  will be studied in Appendix A and it is true for most radii $a$ and $b$.
 Let us introduce the functions:
 $$ U_{\xi,\lambda}(y)=C_N\lambda^{\frac{N-2}{2}}\left(\frac{1}{1+\lambda^{2}|y-\xi|^2}\right)^{\frac{N-2}{2}},\  \xi,y\in\mathbb{R}^N,\ \lambda>0 $$
 which are all the positive solutions of the following critical problem on the whole space:
\begin{equation}\label{eqrn}
-\Delta U= U^{N+2\over N-2}\ \hbox{ in }\mathbb{R}^N,
\end{equation}
where $C_N$ is a constant dependent on $N$ (see \cite{A,O,T}).
We call $U_{\xi,\lambda}(y)$ the  bubble centered at the point $\xi$ with scaling parameter $\lambda.$
  Let us introduce  its projection $P U_{\xi ,\lambda}$ onto $H^1_0(\Omega),$ namely  the solution of the  Dirichlet problem:
\begin{equation}\label{projectedbubbles}
\left\{
\begin{array}{rlll}
-\Delta PU_{\xi , \lambda}&=U_{\xi ,\lambda}^{N+2\over N-2}, &\hbox{in}\ \Omega,\\
PU_{\xi , \lambda}&=0, &\hbox{on } \ \partial\Omega.
\end{array}
\right.
\end{equation}
  Let $k\ge1$ be an integer. Let us choose the centers of the bubbles as the $k$ vertices of a regular $k-$polygon with radius $r$ inside $\Omega$ as:
\begin{equation}\label{point}
\xi_j=r\xi_j^*,\ \xi_j^*:=( e^{\iota {2\pi\over k}j},{\bf{0}}), {\bf{0}}\in\mathbb{R}^{N-2}, j=1, 2, ..., k,\ r\in(a,b)
 \end{equation}
and the concentration parameter as:
\begin{equation}\label{para}
\lambda=\ell k^2,\ \ell\in [\eta ,\eta^{-1}  ]\ \hbox{for some $\eta>0$ small enough.}\end{equation}
    Finally, we introduce the space
      $$\begin{aligned}H_{s}:=&\left\{
u\in H_0^{1}(\Omega)\ :\right.\\ &  \ u(x_1,x_2, \dots,x_i,\dots,x_N)=u(x_1,x_2,\dots,-x_i,\dots,x_N),\ i=2,\dots,N,\\
&\left.
u(re^{\iota\theta},x_3,\dots,x_N)=u\left(re^{\iota\left(\theta+{2\pi\over k}j\right)},x_3,\dots,x_N\right),\ j=1,\dots, k\right\}\end{aligned}$$

Now, we can state our main result.
\begin{theorem}\label{thm:scaled} Let $\Omega:=\{x\in \mathbb{R}^N: a<|x|<b\}$ be an annulus. Assume  \eqref{nondegene}.
Then there exists an integer $k_0>0$, such that for any integer $k\geq k_0$,  problem \eqref{main}  has a solution
$$u_k(x)=u_0(x)-\sum\limits_{j=1}^k PU_{r_k\xi_j^*, \lambda_k}(x)+\varphi_k(x).$$
Where as $k\to\infty$
\begin{itemize}
\item[(i)] $r_k\to r_0\in(a,b)\ \hbox{and}\ r_0^{\frac{N-2}{2}} u_0(r_0):=\max\limits_{a\le r\le b}r^{N-2\over2} u_0(r)$
\item[(ii)] $\lambda_k/k^2\to \ell_0>0$
 \item[(iii)] $\varphi_k\in H_s$ and $\|\varphi_k\|_{H^1_0(\Omega)}\to0$
\end{itemize}
\end{theorem}

      The paper is inspired by  recent results  obtained by Del Pino, Musso, Pacard and Pistoia \cite{dmpp1,dmpp2}, where the authors
 constructed for any $N\geq 3$ infinitely many sign-changing solutions to \eqref{eqrn} which look  like the solution $U_{0, 1}$ crowned with $k$ negative bubbles arranged on a regular polygon with radius near $1.$
 
   \bigskip
   
 For the proof of our theorem, it relies on a Ljapunov-Schmidt procedure which allows us to reduce the problem of finding a solution to \eqref{main} whose profile at main order is $u_0-\sum\limits_{j=1}^k PU_{r\xi_j^*,\lambda}$ to a $2-$dimensional problem, namely finding the concentration parameter $\lambda>0$ in \eqref{para} and  the radius $r\in (a,b)$ of the $k-$regular polygon whose vertices are the concentration points as in \eqref{point}.
The basic outline is similar to that in \cite{dmpp1}, but we carry out  the reduction argument in a different way.  Indeed,  the invariance  by Kelvin's transform which is one of the main ingredient in the proof of \cite{dmpp1}, does not hold for problem \eqref{main}. In particular, all our estimates are more straightforward than those used in \cite{dmpp1}.
 \bigskip

This paper is organized as follows. In Section \ref{s2}  we study the linearized equation around the approximate solution and we reduce the problem to  a finite dimensional one. In Section \ref{s3} we study the reduced problem and we complete the proof of Theorem \ref{thm:scaled}.
 Appendix A is devoted to the study  of the non-degeneracy of  the positive radial solution  $u_0$.

\section{Finite-dimensional reduction}\label{s2}

Let us  introduce the   norms:
\begin{equation}\label{2.1}
\|u\|_{ *}=\sup_{y\in\R^{N}}\Big(\sum_{j=1}^{k}\frac{1}{(1+\lambda|y-\xi_{j}|)^{\frac{N-2}2+\tau}}\Big)^{-1}\lambda^{-\frac{N-2}{2}}|u(y)|
\end{equation}
and
\begin{equation}\label{2.2}
\|f\|_{ **}=\sup_{y\in\R^{N}}\Big(\sum_{j=1}^{k}\frac{1}{(1+\lambda|y-\xi_{j}|)^{\frac{N+2}2+\tau}}\Big)^{-1}\lambda^{-\frac{N+2}{2}}|f(y)|,
\end{equation}
where $\tau=\frac12$.  Since we assume that $\lambda\sim k^2$, it holds

\[
\sum_{j=2}^k \frac1{|\lambda \xi_j-\lambda \xi_1|^\tau}\le \frac{Ck}{\lambda^\tau}\le C.
\]
Set $U_j=:  U_{\xi_j,\lambda}(y),$ $ P_j=:P U_{\xi_j,\lambda}(y) $ and
 $U_*=u_0-\sum^k_jP_j.$

Denote
\begin{eqnarray*}
Z_{j,1}=\frac{\partial P_j }{\partial \lambda},\,\,\,\,\,
Z_{j,2}=\frac{\partial P_j}{\partial
r}, \ \ j=1, 2, ..., k.
\end{eqnarray*}

We consider the following linearized problem:
 \begin{equation}\label{2.3}
\left\{\begin{split}
L_k \varphi :=&-\Delta \varphi
-(2^{*}-1)|U_*|^{2^{*}-2}\varphi
=
h
+\ds\sum_{l=1}^{2}c_{l}\ds\sum_{j=1}^{k} U_j^{2^{*}-2}Z_{j,l},\,\,\,\text{in}\,\,\,\,\Omega,\\
&\varphi\in H_{s},\,\,\, \sum\limits_{j=1}^{k} \ds\int_{\Omega} U_j^{2^{*}-2}Z_{j,l}\varphi=0,\;l=1,2,
 \end{split}\right.
\end{equation}
for some real numbers $c_{l}$.

 \begin{lemma}\label{lem2.1}
 Suppose that $\varphi_{k}$ solves \eqref{2.3} for $h=h_{k}$. If $\|h_{k}\|_{ **}$ goes to zero as $k\to +\infty$, so does $\|\varphi_{k}\|_{ *}$.
 \end{lemma}

 \begin{proof}

 We   argue by contradiction. Suppose  that there exist $k\rightarrow
+\infty,\, r_k\to r_{0},  \lambda_{k}\in
[L_{0}k^{2},L_{1}k^{2}]$ and
$\varphi_{k}$ solving \eqref{2.3} for
$h=h_{k},\lambda=\lambda_{k}, r= r_{k}$ with
$\|h_{k}\|_{**}\rightarrow 0$ and $\|\varphi_{k}\|_{*}\geq c>0.$  Without loss of generality, we
may assume that $\|\varphi_{k}\|_{*}=1$. In the following, for simplicity reason, we drop the
subscript $k$.

Since we assume $u_0$ is non-degenerate, the following linear operator:

\[
\tilde{L}_0 \varphi := -\Delta \varphi - (2^*-1) u_0^{2^*-2}\varphi,\quad \varphi\in H^1_0(\Omega),
\]
is invertible. Let $G(y, x)$ be the corresponding Green's function. It is easy to prove that there exists a constant $C>0$, such that
\begin{equation}\label{green}|G(y, x)|\le \frac C{|y-x|^{N-2}}. \end{equation}  We rewrite  \eqref{2.3} as:

\begin{equation}\label{n-2.3}
\left\{\begin{array}{ll}
L_0 \varphi
=&(2^{*}-1)\bigl( |U_*|^{2^{*}-2}- u_0^{2^*-2}\bigr)\varphi+
h
+\ds\sum_{l=1}^{2}c_{l}\ds\sum_{j=1}^{k} U_j^{2^{*}-2}Z_{j,l},\,\,\,\text{in}\,\,\,\,\Omega,\vspace{0.15cm}\\
&\varphi\in H_{s},\,\,\, \sum\limits_{j=1}^{k} \ds\int_{\Omega} U_j^{2^{*}-2}Z_{j,l}\varphi=0,\;l=1,2.
 \end{array}\right.
\end{equation}
Then

\[
\varphi(y)= \int_{\Omega} G(z, y)\Bigl[ (2^{*}-1)\bigl( |U_*|^{2^{*}-2}- u_0^{2^*-2}\bigr)\varphi+
h
+\ds\sum_{l=1}^{2}c_{l}\ds\sum_{j=1}^{k} U_j^{2^{*}-2}Z_{j,l}\Bigr].
\]
Using \eqref{green}, we obtain

\[
|\varphi(y)|\le C\int_{\Omega} \frac1{|z-y|^{N-2}}\Bigl|\Bigl[ (2^{*}-1)\bigl(| U_*|^{2^{*}-2}- u_0^{2^*-2}\bigr)\varphi+
h
+\ds\sum_{l=1}^{2}c_{l}\ds\sum_{j=1}^{k} U_j^{2^{*}-2}Z_{j,l}\Bigr]\Bigr|.
\]

As in \cite{wy},
 we have
\begin{equation}\label{2.5}
\begin{split}
& \int_{\Omega} \frac1{|z-y|^{N-2}}| |U_*|^{2^{*}-2}- u_0^{2^*-2}||\varphi|\\
\le &C\int_{\Omega} \frac1{|z-y|^{N-2}} (\sum_{j=1}^k P_j)^{2^*-2}|\varphi|\\
\le & C\int_{\mathbb R^N } \frac1{|z-y|^{N-2}} (\sum_{j=1}^k U_j)^{2^*-2}|\varphi|\\
\le & C\|\varphi\|_{*}
\int_{\mathbb R^N } \frac1{|z-y|^{N-2}} (\sum_{j=1}^k U_j)^{2^*-2}\sum_{j=1}^{k}\frac{\lambda^{\frac{N-2}{2}}}{(1+\lambda|z-\xi_{j}|)^{\frac{N-2}2+\tau}}
 \\
\leq  &C\|\varphi\|_{*}\lambda^{\frac{N-2}{2}}\ds\sum_{j=1}^{m}\frac{1}{(1+\lambda|y-\xi_{j}|)^{\frac{N-2}2+\tau+\theta}}.
 \end{split}
\end{equation}

\begin{equation}\label{2.6}
\begin{split}
&\int_{\Omega} \frac{1}{|z-y|^{N-2}}|h(z)|dz\\
\le & C\|h\|_{**}\int_{\mathbb R^N} \frac{1}{|z-y|^{N-2}}\sum_{j=1}^{k}\frac{\lambda^{\frac{N+2}{2}}}{(1+\lambda|z-\xi_{j}|)^{\frac{N+2}2+\tau}}dz
\\
\leq & C\|h\|_{\alpha,**}\lambda^{\frac{N-2}{2}}\ds\sum_{j=1}^{k}\frac{1}{(1+\lambda|y-\xi_{j}|)^{\frac{N-2}2+\tau}},
 \end{split}
\end{equation}
and
\begin{equation}\label{2.7}
\begin{split}
&\ds\int_{\Omega}
\frac{1}{|z-y|^{N-2}}\Big|\ds\sum_{j=1}^{k} U_j^{2^{*}-2}Z_{j,l}\Big|dz\\
\le &
C\lambda^{\frac{N+2}{2}+n_{l}}\int_{\mathbb R^N }
\frac{1}{|z-y|^{N-2}} \sum_{j=1}^{k} \frac{1}{(1+\lambda|z-\xi_{j}|)^{N+2}}
\\
\leq &
C\lambda^{\frac{N-2}{2}+n_{l}}\ds\sum_{j=1}^{k}\frac{1}{(1+\lambda|y-\xi_{j}|)^{N-2}}\\
\le& C\lambda^{\frac{N-2}{2}+n_{l}}\ds\sum_{j=1}^{k}\frac{1}{(1+\lambda|y-\xi_{j}|)^{\frac{N-2}2+\tau}},
 \end{split}
\end{equation}
where $n_2=1$,  $n_{1}=-1.$

To estimate $c_{l},l=1,2$, multiplying the both sides of \eqref{2.3} by the function $Z_{1,l},\; (l=1,2)$ and integrating on $\Omega$, we see that $c_{l}$ satisfies:
\begin{equation}\label{2.8}
\begin{split}
&\ds\sum_{h=1}^{2}c_{h}\sum_{j=1}^{k}\int_{\Omega } U_j^{2^{*}-2}Z_{j,h}Z_{1,l}
\\
=&
\int_{\Omega }\Bigl(-\Delta \varphi
-(2^{*}-1) |U_*|^{2^{*}-2}
\varphi\Bigr)Z_{1,l}-\int_{\Omega } h Z_{1,l}.
 \end{split}
 \end{equation}

We  have

\begin{equation}\label{2.8'}
 \begin{split}
&\big|\int_{\Omega } h Z_{1,l}\big|\\
\leq &C\|h\|_{**}\ds\int_{\mathbb R^N}\frac{\lambda^{\frac{N-2}{2}+n_{l}}}{(1+\lambda|z-\xi_{1}|)^{N-2}}
\sum_{j=1}^{k}\frac{\lambda^{\frac{N+2}{2}}}{(1+\lambda|z-\xi_{j}|)^{\frac{N-2}2+\tau}}\\
\leq & C \lambda^{n_{l}}\|h\|_{ **}\Big( C + C\sum_{j=2}^{k}\frac1{(\lambda|\xi_j-\xi_1|
)^\tau}\Big)\leq C \lambda^{n_{l}}\|h\|_{ **}.
 \end{split}
\end{equation}

On the other hand, direct calculation gives
\begin{equation}\label{2.9.5}
\begin{split}
&\Big|\int_{\Omega }\Bigl(-\Delta \varphi
-(2^{*}-1) |U_*|^{2^{*}-2}
\varphi\Bigr)Z_{1,l}
  \Big|
  \\
  =&\Big|\int_{\Omega }\Bigl(-\Delta Z_{1,l}
-(2^{*}-1)| U_*|^{2^{*}-2}Z_{1,l}\Bigr)
\varphi
  \Big|\\
  =&(2^{*}-1)\Big|\int_{\Omega } \bigl(U_1^{2^*-2} - |U_*|^{2*-2}\bigr)Z_{1,l}\Big|
\varphi
  \\
  \le &  C\lambda^{n_{l}}\|\varphi\|_{ *}\int_{\Omega }\Bigl( u_0^{2^*-1} + \bigl(\sum_{j=2}^k  U_j\bigr)^{2^*-2}\Bigr) U_1
  \sum_{j=1}^{k}\frac{\lambda^{\frac{N-2}{2}}}{(1+\lambda|z-\xi_{j}|)^{\frac{N-2}2+\tau}}
  \\
  \le &O\Big(
\lambda^{n_{l}}\|\varphi\|_{*}\Bigl( \frac{1}{{\lambda^{2}}}+ \frac{1}{{\lambda^{\frac{N-2}2}}}\Bigr)\Big).
\end{split}
 \end{equation}

And it is easy to check that

 \begin{equation}\label{2.14'}
 \sum_{j=1}^{k} \int_{\Omega } U_j^{2^{*}-2}Z_{j,h}Z_{1,l}
= ( \bar{c}  +o(1)) \delta_{hl}\lambda^{2n_l},
 \end{equation}
 for some constant $\bar{c}>0$.

Now inserting  \eqref{2.14'} into
 \eqref{2.8}, we find
 \begin{equation}\label{2.16}
 c_{l}=\frac{1}{\lambda^{n_{l}}}\big(o(\|\varphi\|_{\alpha, *})+O(\|h\|_{\alpha, **})\big).
 \end{equation}
So,
 \begin{equation}\label{2.11}
\|\varphi\|_{ *}\leq
\Bigg(o(1)+\|h\|_{ **}+\frac{\sum_{j=1}^{k}\frac{1}{(1+\lambda|y-\xi_{j}|)^{\frac{N-2}2+\tau+\theta}}}
{\sum_{j=1}^{k}\frac{1}{(1+\lambda|y-\xi_{j}|)^{\frac{N-2}2+\tau}}}\Bigg).
 \end{equation}

Since $\|\varphi\|_{*}=1$, we obtain from \eqref{2.11} that there is $R>0$ such that
 \begin{equation}\label{2.12}
 \|\lambda^{\frac{N-2}{2}}\varphi\|_{L^\infty(B_{R/\lambda}(\xi_{j}))}\geq a>0,
 \end{equation}
 for some $j$. But $\tilde{\varphi}(y)=\lambda^{-\frac{N-2}{2}}\varphi(\lambda(y-x_{j}))$ converges uniformly in any compact set to a solution $u$ of
\begin{equation}\label{2.13}
-\Delta  u-(2^{*}-1)U_{0,\Lambda}^{2^{*}-2}u =0,\,\,\,\text{in}\,\,\,
\R^{N},
 \end{equation}
 for some $\Lambda\in[\Lambda_{1},\Lambda_{2}],$ where $\Lambda_1, \Lambda_2$ are two constants, and $u$ is perpendicular to the kernel of \eqref{2.13}. So $u=0$. This is a contradiction to \eqref{2.12}.

 \end{proof}

From Lemma \ref{lem2.1}, applying the same argument as in the proof of Proposition 4.1 in \cite{dfm}, we can
prove the following result:
 \begin{lemma}\label{lem2.2}
 There exist $k_{0}>0$ and a constant $C>0$ independent of $k$, such that for $k\geq k_{0}$ and
 all $h\in L^{\infty}(\R^{N})$, problem \eqref{2.3} has a unique solution $\varphi_k \equiv L_{k}(h)$. Moreover,
 \begin{equation}\label{2.10}
 \|\varphi_k\|_{ *}\leq C\|h\|_{ **},\,\,\,\,|c_{l}|\leq \frac{C}{\lambda^{n_l}}\|h\|_{**}.
 \end{equation}
  \end{lemma}

  Now we consider the following non-linear problem:
 \begin{equation}\label{2.21}
\left\{\begin{array}{ll}
&-\Delta (U_*+\varphi)
=
|U_*+\varphi|^{2^{*}-2}(U_*+\varphi)
+\ds\sum_{l=1}^{2}c_{l}\ds\sum_{j=1}^{k} U_j^{2^{*}-2}Z_{j,l},\,\,\,\,\,\,\text{in}\,\,\Omega,\vspace{0.15cm}\\
&\varphi\in H_{s},\,\,\,\ds\int_{\Omega}
\sum\limits_{j=1}^{k} U_j^{2^{*}-2}Z_{j,l}\varphi=0,\;l=1,2.
 \end{array}\right.
\end{equation}
The main  result of this section is:

\begin{proposition}\label{prop2.3}
There exists a positive integer $k_{0}$ such that for each $k\geq
k_{0},\lambda\in[\eta k^{2},\eta^{-1}k^{2}],  r\in[a+\tau,b-\tau]$, where $\tau$ and $\eta$ are positive and  small, \eqref{2.21} has a
unique solution $\varphi=\varphi_{ r,,\lambda}\in
H_{s}$ satisfying
\begin{equation}\label{2.22}
\|\varphi\|_{*}\leq
 C
\lambda^{ -\frac {N-2}4-\sigma   },\,\,\,\,\, |c_{l}|\leq
 C
\lambda^{ -\frac {N-2}4-\sigma -n_l  },
\end{equation}
where $\sigma>0$ is a small constant.
\end{proposition}

Rewrite \eqref{2.21} as:
\begin{equation}\label{l}
\left\{\begin{array}{ll}
&-\Delta \varphi
-(2^*-1)
|U_*|^{2^{*}-2}
\varphi=N(\varphi)+l_{k}+\ds\sum_{l=1}^{2}c_{l}\ds\sum_{j=1}^{k} U_j^{2^{*}-2}Z_{j,l},\,\,\,\,\,\,\text{in}\,\,\Omega,\\
&\varphi\in H_{s},\,\,\,\ds\int_{\Omega}
\sum\limits_{j=1}^{k} U_j^{2^{*}-2}Z_{j,l}\varphi=0,\;l=1,2,
 \end{array}\right.
\end{equation}
where
\[
 N(\varphi) =
|U_*+\varphi|^{2^{*}-2}(U_*+\varphi)- |U_*|^{2^{*}-2}U_*-(2^*-1)
|U_*|^{2^{*}-2}
\varphi,
\]
and
\[
l_{k} =  |U_*|^{2^{*}-2}U_*-u_0^{2*-1}
+\sum_{j=1}^{k} U_j^{2^{*}-1}.
\]

In order to apply the contraction mapping principle to prove that
\eqref{l} is uniquely solvable, we have to estimate $N(\varphi)$ and
$l_{k}$ respectively.

\begin{lemma}\label{lem2.4}
We have

$$
|| N(\varphi)||_{**}\leq C\|\varphi\|_{*}^{\min(2^{*}-1,2)}.
$$
\end{lemma}

\begin{proof}

If $N\ge 6$, then $2^*-2\le 1$. So we have
$$
|N(\varphi)|\leq
C|\varphi|^{2^{*}-1},
 $$
which gives

\[
 \begin{split}
 |N(\varphi)| \leq & C\|\varphi\|_{ *}^{2^{*}-1}
 \Big(\sum_{j=1}^{k}\frac{\lambda^{\frac{N-2}{2}}}{(1+\lambda|y-\xi_{j}|)^{\frac{N-2}2+\tau}}\Big)^{2^{*}-1} \\
\leq&
 C\|\varphi\|_{*}^{2^{*}-1}\lambda^{\frac{N+2}{2}}\sum_{j=1}^{k}\frac{1}{(1+\lambda|y-\xi_{j}|)^{\frac{N+2}{2}+\tau}}
 \Big(\sum_{j=1}^{k}\frac{1}{(1+\lambda|y-\xi_{j}|)^{\tau}}\Big)^{\frac{4}{N-2}}\\
 \leq&  C\|\varphi\|_{*}^{2^{*}-1}\lambda^{\frac{N+2}{2}}\sum_{j=1}^{k}\frac{1}{(1+\lambda|y-\xi_{j}|)^{\frac{N+2}{2}+\tau}}.
 \end{split}
 \]
 Therefore,
 $$
 \|N(\varphi)\|_{**}\leq C\|\varphi\|_{*}^{2^{*}-1}.
 $$

Similarly, if $3\le N\le 5$, then $2^*-3>0$. In view of $U_j\ge c_0>0$ in $\Omega$, we find
\begin{eqnarray*}
|N(\varphi)|&\leq& C \bigl(|u_0|^{2^*-3}+(\sum_{j=1}^k U_j)^{2^*-3}\bigr)
\varphi^{2}+
C|\varphi|^{2^*-1}
\\
&\leq&
 C(\|\varphi\|_{*}^{2}+\|\varphi\|_{*}^{2^*-1})\lambda^{\frac{N+2}{2}}
 \Big(\sum_{j=1}^{k}\frac{1}{(1+\lambda|y-\xi_{j}|)^{\frac{N-2}{2}+\tau}}\Big)^{2^*-1}\\
 &\leq&  C\|\varphi\|_{*}^{2}\sum_{j=1}^{k}\frac{\lambda^{\frac{N+2}{2}}}{(1+\lambda|y-\xi_{j}|)^{\frac{N+2}{2}+\tau}}.
\end{eqnarray*}

\end{proof}

Next, we estimate $l_{k}.$
\begin{lemma}\label{lem2.5}
There is a constant $\sigma>0$, such that
$$
\|l_{k}\|_{**}\leq C
\lambda^{ -\frac {N-2}4-\sigma   }.
$$

\end{lemma}

\begin{proof}

Write

\[
\begin{array}{ll}
l_{k} &= \Bigl[ |U_*|^{2^{*}-2}U_*-u_0^{2^*-1}
+\sum_{j=1}^{k} P_j^{2^{*}-1} \Bigr] + \sum_{j=1}^{k}\bigl( U_j^{2^*-1}- P_j^{2^{*}-1}\bigr)\\
&=: J_1+ J_2.
\end{array}
\]

First, we estimate  $\|J_2\|_{**}$.  We have

\[
0\le  U_j^{2^*-1}- P_j^{2^{*}-1}\le \frac{C U_j^{2^*-2}}{\lambda^{\frac{N-2}2}}.
 \]
Let us determine the number $\alpha>0$, such that

\[
\frac{C U_j^{2^*-2}}{\lambda^{\frac{N-2}2}}\le \frac{ C\lambda^{-\alpha}\lambda^{\frac{N+2}{2}}}{(1+\lambda|y-\xi_{j}|)^{\frac{N+2}{2}+\tau}}.
\]
The above inequality is equivalent to

\[
(1+\lambda|y-\xi_{j}|)^{\frac{N+2}{2}+\tau-4 }\le  C\lambda^{-\alpha}\lambda^{N-2}.
\]
Note that $\tau=\frac12$. We find that $\frac{N+2}{2}+\tau-4\ge 0$ if $N\ge 5$. In view of $1+\lambda|y-x_{j}|\le C\lambda$ in $\Omega$. We have

\[
(1+\lambda|y-\xi_{j}|)^{\frac{N+2}{2}+\tau-4 }\le C\lambda^{\frac{N+2}{2}+\tau-4 }= C  \lambda^{- \frac{N-1}2  } \lambda^{N-2}.
\]
As an result,  $\alpha=\frac{N-1}2$.  Thus, we get

\begin{equation}\label{n1-28-2}
\| J_2\|_{**}\le C\lambda^{- \frac{N-1}2  }, \quad \hbox{ if } N\ge 5.
\end{equation}

If $N\le 5$, it holds $\frac{N+2}{2}+\tau-4<0$. Thus

\[
(1+\lambda|y-\xi_{j}|)^{\frac{N+2}{2}+\tau-4 }\le C = C\lambda^{2-N}\lambda^{N-2}.
\]
So $\alpha=N-2$. Hence, we obtain

\begin{equation}\label{n2-28-2}
\| J_2\|_{**}\le C\lambda^{2-N  }, \quad \hbox{ if }N\le 5.
\end{equation}

In order to  estimate  $\|J_1\|_{**}$. We define
$$
\Omega_{j}=\Big\{y:y=(y',y'')\in \R^{2}\times \R^{N-2},\Big\langle \frac{y'}{|y'|},\frac{\xi'_{j}}{|\xi_{j}|}\Big\rangle\geq \cos\frac{\pi}{k}\Big\}.
$$
Using the assumed symmetry, we just need to estimate $J_1$ in $\Omega_1$.  Let  $S= \Omega_1\cap B_{1/\sqrt\lambda}(\xi_1)$.

Note that, it holds $P_1\ge c_0>0$ in $S$, and
$$|U_*|^{2^*-2}U_*=|u_0-\sum_{j=2}^k P_j-P_1|^{2^*-2}(u_0-\sum_{j=2}^k P_j-P_1),$$ we have
\[
|J_1|\leq  P_1^{2^*-2} \bigl( u_0 +\sum_{j=2}^k P_j) + J_3,
\]
where $|J_3|\le C$ in $S$.

Since
\[
\frac{\lambda^{\frac{N+2}2}}{ (1+ \lambda|y-\xi_1|)^{\frac{N+2}2+\tau}}\ge \frac{\lambda^{\frac{N+2}2}}{ (1+ \sqrt\lambda)^{\frac{N+2}2+\tau}}\ge a_0\lambda^{\frac{N+2}4-\frac\tau 2},\quad y\in S,
\]
it holds

\[
|J_3| \le C\lambda^{-\frac{N+2}4+\frac\tau 2}\frac{\lambda^{\frac{N+2}2}}{ (1+ \lambda|y-\xi_1|)^{\frac{N+2}2+\tau}},\quad y\in S.
\]

On the other hand

\[
|P_1^{2^*-2} \bigl( u_0 +\sum_{j=2}^k P_j)|\le C U_1^{2^*-2},
\]
and if $N\ge 5$,

\[
(1+\lambda|y-\xi_1|)^{\frac{N+2}2+\tau -4}\le C\lambda^{\frac12 (\frac{N+2}2+\tau -4)},
\]
which  gives

\[
|P_1^{2^*-2} \bigl( u_0 +\sum_{j=2}^k P_j)|\le C U_1^{2^*-2}\le\lambda^{ -\frac {N+2}4+\frac\tau2   } \frac{\lambda^{\frac{N+2}2}}{ (1+ \lambda|y-\xi_1|)^{\frac{N+2}2+\tau}},\quad y\in S.
\]

If $N=3, 4$,
\[
(1+\lambda|y-\xi_1|)^{\frac{N+2}2+\tau -4}\le C,
\]
which  gives

\[
|P_1^{2^*-2} \bigl( u_0 +\sum_{j=2}^k P_j)|\le C U_1^{2^*-2}\le\lambda^{ -\frac {N-2}2  } \frac{\lambda^{\frac{N+2}2}}{ (1+ \lambda|y-\xi_1|)^{\frac{N+2}2+\tau}},\quad y\in S.
\]
Therefore, we have proved

\begin{equation}\label{99-3-3}
|J_1|\le C \lambda^{ -\frac {N-2}4-\sigma   } \frac{\lambda^{\frac{N+2}2}}{ (1+ \lambda|y-\xi_1|)^{\frac{N+2}2+\tau}},\quad y\in S.
\end{equation}

On the other hand, we note that, in $\Omega_1\setminus S$, it holds $P_1\le C$. Thus

\[
\begin{split}
|J_1 |\le &C \sum_{j=1}^k U_j\\
\le & \frac{C}{\lambda^{\frac{N-2}2}|y-\xi_1|^{N-2}}+ \frac{C}{\lambda^{\frac{N-2}2}|y-\xi_1|^{N-2-\tau}}\sum_{j=2}^k\frac1{|\xi_j-\xi_1|^\tau}\\
\le &\frac{C}{\lambda^{\frac{N-2}2-\tau}|y-\xi_1|^{N-2-\tau}}.
\end{split}
\]
Now we determine $\beta>0$, such that

\begin{equation}\label{100-3-3}
\frac{1}{\lambda^{\frac{N-2}2-\tau}|y-\xi_1|^{N-2-\tau}}\le C \lambda^{-\beta}\frac{\lambda^{\frac{N+2}2}}{ (1+ \lambda|y-\xi_1|)^{\frac{N+2}2+\tau}}, \quad y\in  \Omega_1\setminus S.
\end{equation}

It holds

\[
\frac{\lambda^{\frac{N+2}2}}{ (1+ \lambda|y-\xi_1|)^{\frac{N+2}2+\tau}}\ge \frac{c'}{  \lambda^\tau |y-\xi_1|^{\frac{N+2}2+\tau}}, \quad y\in  \Omega_1\setminus S.
\]
So \eqref{100-3-3} holds if

\[
\frac{1}{\lambda^{\frac{N-2}2-\tau}|y-\xi_1|^{N-2-\tau}}\le \frac{C\lambda^{-\beta}}{  \lambda^\tau |y-\xi_1|^{\frac{N+2}2+\tau}}, \quad y\in  \Omega_1\setminus S.
\]
which is equivalent to

\begin{equation}\label{101-3-3-0}
C|y-\xi_1|^{  N-2-2\tau -\frac{N+2}2 } \ge  \lambda^{ \beta +2\tau - \frac{N-2}2 }, \quad y\in  \Omega_1\setminus S.
\end{equation}
Since $|y-\xi_1|\ge \frac1{\sqrt\lambda}$, we can take

\[
\beta=\frac{N-2}2-2\tau -\frac12 ( N-2-2\tau -\frac{N+2}2)=\frac{N+2}4-\tau,
\]
if $N-2-2\tau -\frac{N+2}2\ge 0$. That is $N\ge 8$.  If $N\le 8$, we can take $\beta= \frac{N-2}2-2\tau$.

So, we have proved

\begin{equation}\label{101-3-3}
|J_1|\le C \lambda^{-\frac{N-2}4-\sigma}\frac{\lambda^{\frac{N+2}2}}{ (1+ \lambda|y-\xi_1|)^{\frac{N+2}2+\tau}}, \quad y\in  \Omega_1\setminus S.
\end{equation}

Combining \eqref{99-3-3} and \eqref{101-3-3}, we find that there exists $\sigma>0$, such that

\begin{equation}\label{102-3-3}
|J_1|\le C \lambda^{ -\frac {N-2}4-\sigma   } \frac{\lambda^{\frac{N+2}2}}{ (1+ \lambda|y-\xi_1|)^{\frac{N+2}2+\tau}},\quad y\in \Omega_1.
\end{equation}
This gives

\[
\|J_1\|_{**}\le  C \lambda^{ -\frac {N-2}4-\sigma   }.
\]

\end{proof}

Now we are ready to the proof of Proposition~\ref{prop2.3}.

\begin{proof} [ Proof  of  Proposition~\ref{prop2.3}.]

First we recall that $\lambda\in[\eta k^2 ,\eta^{-1} k^2] $ for some $\eta>0.$
Set
$$
\mathcal {N}=\Big\{w:w\in
C(\R^{N})\cap H_{s},\|w\|_{*}\leq \frac{1}{\lambda^{\frac{N-2}4}}, \int_{\Omega} \sum_{j=1}^k U_j^{2^{*}-2}Z_{j,l} w=0\Big\},
$$
where $l=1,2.$
Then \eqref{l} is equivalent to

\begin{equation}\label{A}
\varphi=\mathcal
{A}(\varphi)=:L_{k}(N(\varphi))+L_{k}(l_{k}),
\end{equation}
here $L_{k}$ is defined in Lemma \ref{lem2.2}.
We will prove that $\mathcal {A}$ is a contraction map from
$\mathcal {N}$ to $\mathcal {N}.$

First, we have
$$
\begin{array}{ll}
\|\mathcal {A}(\varphi)\|_{*}&\leq
C\|N(\varphi)\|_{**}+C\|l_{k}\|_{**}\\
&\leq
C\|\varphi\|_{*}^{\min\{2^{*}-1,2\}}+C\frac{1}{\lambda^{\frac{N-2}4+\sigma}}\\
 &\leq
\frac{1}{\lambda^{\frac{N-2}4}}.
\end{array}
$$
Hence, $\mathcal {A}$ maps $\mathcal {N}$ to $\mathcal {N}.$

On the other hand, we see
\begin{eqnarray*}
\begin{array}{ll}
\|\mathcal {A}(\varphi_{1})-\mathcal
{A}(\varphi_{2})\|_{*}&=\|L_{k}(N(\varphi_{1}))-L_{k}(N(\varphi_{2}))\|_{*}\\
&\leq C\|N(\varphi_{1})-N(\varphi_{2})\|_{**}.\end{array}
\end{eqnarray*}
It is easy to check that if $N\geq 6,$ then
\begin{eqnarray*}
&&|N(\varphi_{1})-N(\varphi_{2})|\\
&\leq& |N'(\varphi_{1}+\theta\varphi_{2})||\varphi_{1}-\varphi_{2}|\\
&\leq &C
(|\varphi_{1}|^{2^{*}-2}+|\varphi_{2}|^{2^{*}-2})
|\varphi_{1}-\varphi_{2}|\\
&\leq&C
(||\varphi_{1}||_{*}^{2^{*}-2}+||\varphi_{2}||_{*}^{2^{*}-2})
||\varphi_{1}-\varphi_{2}||_{*}\Big(\sum_{j=1}^{k}\frac{\lambda^{\frac{N-2}{2}}}{(1+\lambda|y-\xi_{j}|)^{\frac{N-2}{2}+\tau}}\Big)^{2^{*}-1}.
\end{eqnarray*}
As before, we have
$$
\Big(\sum_{j=1}^{k}\frac{1}{(1+\lambda|y-\xi_{j}|)^{\frac{N-2}{2}+\tau}}\Big)^{2^{*}-1}\leq
C\sum_{j=1}^{k}\frac{1}{(1+\lambda|y-\xi_{j}|)^{\frac{N+2}{2}+\tau}}.
$$
Hence,
\begin{eqnarray*}\begin{array}{ll}
\|\mathcal {A}(\varphi_{1})-\mathcal
{A}(\varphi_{2})\|_{*}
&\leq C(\|\varphi_{1}\|_{*}^{2^{*}-2}+\|\varphi_{2}\|_{*}^{2^{*}-2})\|\varphi_{1}-\varphi_{2}\|_{*}\\
&\leq\frac{1}{2}\|\varphi_{1}-\varphi_{2}\|_{*}.\end{array}
\end{eqnarray*}
Therefore, $\mathcal {A}$ is a contraction map.

The case $N\le 5$ can be proved in a similar way.

\vskip8pt
Now by using the contraction mapping theorem, there exists a unique
$\varphi=\varphi_{ r,\lambda}\in \mathcal {N}$ such that
\eqref{A} holds. Moreover, by Lemmas \ref{lem2.2}, \ref{lem2.4} and \ref{lem2.5},  we deduce
\begin{eqnarray*}\begin{array}{ll}
\|\varphi\|_{*}&\leq\|L_{k}(N(\varphi))\|_{*}+\|L_{k}(l_{k})\|_{*}\\
&\leq C\|N(\varphi)\|_{**}+C\|l_{k}\|_{**}\\
&\leq
C\Big(\frac{1}{\lambda}\Big)^{\frac{N-2}4+\sigma}.\end{array}
\end{eqnarray*}
Moreover, we get the estimate of $c_{l}$ from \eqref{2.10}.
\end{proof}

\section{ The Proof of the Main theorem }\label{s3}

 We look
for a solution to \eqref{main} as
 $u=U_*+\varphi ,$   where $ \varphi= \varphi_k$ is the function obtained in Proposition~\ref{prop2.3}.
Let us introduce the energy functional whose critical points are solutions to \eqref{main}
\begin{equation}\label{a1}
I(u)=\frac{1}{2}\int_{\Omega}|\nabla u|^2-\frac{1}{2^{\ast}}\int_{\Omega}|u|^{2^{\ast}}.
\end{equation}
and the reduced energy
\begin{equation}\label{red-ene}
  I_k(\ell,r):=I(U_*+\varphi).\end{equation}

Where $$\lambda=\ell k^2,\ \ell\in [\eta ,\eta^{-1}  ]\ \hbox{for some $\eta>0$ small enough.}$$

 We have the following result
 \begin{proposition}\label{prop10}
 \begin{itemize}
 \item[(i)] $U_*+\varphi$  is a critical point of $I$ if and only if $(\ell,r)$ is a critical point of the reduced energy $I_k$
 \item[(ii)] We have
 $$I_k(\ell,r)=I(u_0)+kA +\frac1{k^{N-2}} F(\ell,r)+o\left(\frac1{k^{N-2}}\right)$$
 uniformly in compact sets of $(0,+\infty)\times(a,b)$, where
 $$F(\ell,r):=B{u_0(r)\over\ell^{N-2\over2}}-C{1\over r^{N-2}\ell^{N-2}}$$
 for some positive constants $A,$ $B$ and $C.$
 \end{itemize}
 \end{proposition}

\begin{proof}
The proof of (i) is quite standard. We only prove (ii).
First of all we prove that
\begin{equation}\label{new1}
I(U_*+\varphi)= I(U_*) + k O\Bigl( \lambda^{-\frac{N-2}2 -2\sigma}\Bigr),\ \hbox{for some $\sigma<0.$}
\end{equation}
First of all, we have

\begin{equation}\label{10-2-3}
\begin{split}
I(U_*+\varphi)= &I(U_*) +\frac12\int_\Omega |\nabla \varphi|^2 + \int_\Omega \bigl( u_0^{2*-1} -\sum_{j=1}^k U_j^{2^*-1}\bigr)\varphi\\
&-\frac1{2^*} \int_\Omega \Bigl(  | U_*+\varphi|^{2^*} - | U_*|^{2^*}  \Bigr).
\end{split}
\end{equation}

It follows from \eqref{2.21} that

 \begin{equation}\label{n2-3-2.21}
\int_\Omega |\nabla \varphi|^2
=
\int_\Omega |U_*+\varphi|^{2^{*}-2}(U_*+\varphi)\varphi
-\int_\Omega\Bigl( u_0^{2^-1} -\sum_{j=1}^k U_j^{2^*-1}\Bigr)\varphi.
\end{equation}
Thus, we obtain

\begin{equation}\label{10-3-3}
\begin{split}
&I(U_*+\varphi)\\
= &I(U_*) +\frac12\int_\Omega |U_*+\varphi|^{2^{*}-2}(U_*+\varphi)\varphi + \frac12 \int_\Omega \bigl( u_0^{2*-1} -\sum_{j=1}^k U_j^{2^*-1}\bigr)\varphi\\
&-\frac1{2^*} \int_\Omega \Bigl(  | U_*+\phi|^{2^*} - | U_*|^{2^*}  \Bigr)\\
=&I(U_*)  + \frac12 \int_\Omega \Bigl( u_0^{2*-1} -\sum_{j=1}^k U_j^{2^*-1}-|U_*|^{2^*-2}U_*\Bigr)\varphi\\
&+ \frac12\int_\Omega \Bigl(|U_*+\varphi|^{2^{*}-2}(U_*+\varphi)- |U_*|^{2^{*}-2}U_*  \Bigr)\varphi \\
&-\frac1{2^*} \int_\Omega \Bigl(  | U_*+\varphi|^{2^*} - | U_*|^{2^*} - 2^* |U_*|^{2-2} U_*\varphi \Bigr).
\end{split}
\end{equation}

Write

\[
l_{k} = \Bigl[ |U_*|^{2^{*}-2}U_*-u_0^{2^*-1}
+\sum_{j=1}^{k} P_j^{2^{*}-1} \Bigr] + \sum_{j=1}^{k}\bigl( U_j^{2^*-1}- P_j^{2^{*}-1}\bigr).
\]

 It follows from Lemma~\ref{lem2.5}, there is a constant $\sigma>0$, such that
$$
\|l_{k}\|_{**}\leq C
\lambda^{ -\frac {N-2}4-\sigma   }.
$$

By Proposition~\ref{prop2.3}, we can obtain from \eqref{10-3-3} that if $N\ge 6,$

\begin{equation}\label{n10-4-3}
\begin{split}
&I(U_*+\varphi)\\
=&I(U_*)  + O\Bigl( \|l_k\|_{**}\|\varphi\|_{*} \Bigr)\sum_{j=1}^k \int_\Omega \Bigl(\sum_{j=1}^k\frac{\lambda^{\frac{N+2}2}}{(1+\lambda|y-x_j|)^{\frac{N+2}2+\tau}} \Bigr) \Bigl(\sum_{j=1}^k \frac{\lambda^{\frac{N-2}2}}{(1+\lambda|y-x_j|)^{\frac{N-2}2+\tau}}\Bigr) \\
&+ O\Bigl( \|\varphi\|^{2^*}_{*} \Bigr) \int_\Omega \Bigl(\sum_{j=1}^k \frac{\lambda^{\frac{N-2}2}}{(1+\lambda|y-x_j|)^{\frac{N-2}2+\tau}}\Bigr)^{2^*}\\
=& I(U_*) + k O\Bigl( \lambda^{-\frac{N-2}2 -2\sigma}\Bigr).
\end{split}
\end{equation}

While if $N\le 5$, then

\begin{equation}\label{n11-4-3}
\begin{split}
&I(U_*+\varphi)\\
=&I(U_*)  + O\Bigl( \|l_k\|_{**}\|\varphi\|_{*} \Bigr)\sum_{j=1}^k \int_\Omega \Bigl(\sum_{j=1}^k\frac{\lambda^{\frac{N+2}2}}{(1+\lambda|y-x_j|)^{\frac{N+2}2+\tau}} \Bigr) \Bigl(\sum_{j=1}^k \frac{\lambda^{\frac{N-2}2}}{(1+\lambda|y-x_j|)^{\frac{N-2}2+\tau}}\Bigr) \\
&+ O\Bigl( \|\varphi\|^{2^*}_{*} \Bigr) \int_\Omega \Bigl(\sum_{j=1}^k \frac{\lambda^{\frac{N-2}2}}{(1+\lambda|y-x_j|)^{\frac{N-2}2+\tau}}\Bigr)^{2^*}\\
&+ O\Bigl( \|\varphi\|^{2}_{*} \Bigr) \int_\Omega |U_*|^{2^*-3}\Bigl(\sum_{j=1}^k \frac{\lambda^{\frac{N-2}2}}{(1+\lambda|y-x_j|)^{\frac{N-2}2+\tau}}\Bigr)^{2}\\
=& I(U_*) + k O\Bigl( \lambda^{-\frac{N-2}2 -2\sigma}\Bigr).
\end{split}
\end{equation}
That concludes the proof of \eqref{new1}.

Next, we prove that
\begin{equation}\label{new2}
I(U_*)= I(u_0)+
k\left[A-\frac{B_1k^{N-2}}{r^{N-2}\lambda^{N-2}}+\frac{B_2u_0(r)}{\lambda^{\frac{N-2}{2}}}+O(\frac{1}{\lambda^{\frac{N-2}{2}(1+\delta)}})\right]
\end{equation}
where $A$, $B_1$, $B_2$ and are positive constants, $\delta>0$ is small.

  Recall that $P_j$ satisfies \eqref{projectedbubbles}
and
set  $V=\sum_{j=1}^k P_j$.  We have

\begin{equation}\label{e1}
\begin{array}{ll}
\quad \displaystyle\int_{\Omega}|\nabla U_*|^2\\
=\displaystyle\int_\Omega |\nabla u_0|^2+\int_\Omega |\nabla V|^2-\displaystyle 2\int_\Omega \nabla V\nabla u_0\\
=\displaystyle\int_\Omega |\nabla u_0|^2+\int_\Omega |\nabla V|^2-\displaystyle 2\int_\Omega u_0^{2^*-1}V.\\
\end{array}
\end{equation}

Let $\Omega_j=:\{(r \cos\theta, r\sin\theta, x')|\frac{2\pi (j-1)}{k}-\frac\pi k\le \theta\leq \frac{2\pi(j)}{k}+\frac\pi k, x'\in \mathbb{R}^{N-2}\}\cap \Omega, j= 1,..., k.$ Then by
the symmetry, we have

$$\int_\Omega u_0^{2^*-1} V =k\int_{\Omega_1} u_0^{2^*-1} V,$$
and

$$\int_\Omega |U_*|^{2^*}=k\int_{\Omega_1} |U_*|^{2^*}.$$

Let $S=:\Omega_1\cap B_{\lambda^{-\frac{1}{2}}}(\xi_1),$ by suing the following inequality:
$$|1-t|^p=1-pt+O(t^2)=1-pt+O(t^\alpha), 1<\alpha\leq 2, \forall 0\leq t\leq c,$$
where $c$ is some constant,
we obtain

\begin{equation}\label{e2}
\begin{array}{ll}
\quad \displaystyle\int_{S}| U_*|^{2^*}\\
=\displaystyle\int_S V^{2^*}-\displaystyle 2^*\int_S V^{2^*-1} u_0+O(\int_S v^{2^*-1-\delta}u_0^\delta)\\
=\displaystyle\int_S V^{2^*}-\displaystyle 2^*\int_S V^{2^*-1} u_0+O(\lambda^{-\frac{(1+\delta)(N-2)}{2}}),\\
\end{array}
\end{equation}
where $\delta>0$ is small.

On the other hand, we have
\begin{equation}\label{e3}
\begin{array}{ll}
\quad \displaystyle\int_{\Omega_1\setminus S}| U_*|^{2^*}\\
=\displaystyle\int_{\Omega_1\setminus S}u_0^{2^*}-\displaystyle 2^*\int_{\Omega_1\setminus S} u_0^{2^*-1} V+O(\int_{\Omega_1\setminus S} u_0^{2^*-2} V^{2})\\
=
\displaystyle\int_{\Omega}u_0^{2^*}-\displaystyle 2^*\int_{\Omega_1\setminus S} u_0^{2^*-1} V+O(\int_{\Omega_1\setminus S} u_0^{2^*-2} V^{2})+O(\lambda^{-\frac{N}{2}}),
\end{array}
\end{equation}
since
\begin{equation}\label{e4}
\begin{array}{ll}
\displaystyle\int_{\Omega_1\setminus S}u_0^{2^*}=\int_{\Omega_1} u_0^{2^*}+O(\lambda^{-\frac{N}{2}}).
\end{array}
\end{equation}

Note that, for any $y\in \Omega_0$, we have

\[
\sum_{j=2}^{k}\frac1{|y-\xi_j|^{N-2}}\le \frac C{|y-\xi_1|^{N-2-\tau }}\sum_{j=2}^{k} \frac1{|\xi_j-\xi_1|^\tau}\le \frac {C k}{|y-\xi_1|^{N-2-\tau }}
\]
for $\tau\in (0, 1)$.
So we obtain

\begin{equation}\label{e5}
\begin{split}
\int_{\Omega_1\setminus S}u_0^{2^*-2}V^{2}\le& C\int_{\Omega_1\setminus S} V^{2}\\
\le &\int_{\Omega_1\setminus S}(\sum_{j=1}^{k}\frac{1}{\lambda^{\frac{N-2}{2}}    |y-\xi_j|^{N-2}})^{2}\\
\le &\frac C{\lambda^{N-2}} \int_{\Omega_1\setminus S}
\Bigl( \frac1{|y-\xi_1|^{N-2}  }+  \frac k{|y-\xi_1|^{N-2-\tau}  }\Bigr)^2
\\
\le &\frac C{\lambda^{N-2}} \int_{\Omega_1\setminus S}
\Bigl( \frac1{|y-\xi_1|^{2(N-2)}  }+  \frac k{|y-\xi_1|^{2(N-2-\tau)}  }\Bigr)
\\
\le &  \frac C{\lambda^{N-2}}
\Bigl(\lambda^{\frac12( N-4  )}+ k^2 \lambda^{\frac12( N-4 -2\tau )}
\Bigr)\\
\le &\frac C{\lambda^{\frac N2- 1+2\tau}},
\end{split}
\end{equation}
since  $k^2 \sim \lambda$.
As a consequence,
\begin{equation}\label{e6}
\begin{split}
&\int_{\Omega_1\setminus S}| U_*|^{2^*}\\
=&\int_{\Omega_1}u_0^{2^*}- 2^*\int_{\Omega_1\setminus S} u_0^{2^*-1} V+ O(\frac{1}{\mu^{\frac{N-2}{2}(1+\delta)}}).
\end{split}
\end{equation}

Combining the above obtained results, we get

\begin{equation}\label{nne6}
\begin{array}{lll}
I(U_*)&=I(u_0)+\displaystyle \frac{1}{2}\int_\Omega|DV|^2-\frac{k}{2^*}\int_S V^{2^*}+k\int_S V^{2^*-1} u_0\\
&\quad-\displaystyle \int_{\Omega} u_0^{2^*-1} V+k\int_{\Omega_1\setminus S}u_0^{2^*-1}V+ O(\frac{k}{\lambda^{\frac{N-2}{2}(1+\delta)}}).
\end{array}
\end{equation}

Now we compute those integrals in \eqref{nne6} one by one:
\begin{equation}\label{e7}
\begin{array}{ll}
\quad -\displaystyle\int_\Omega u_0^{2^*-1}V +k\int_{\Omega_1\setminus S} u_0^{2^*-1}V\\
=\displaystyle -k\int_{\Omega_1} u_0^{2^*-1}V+k\int_{\Omega_1\setminus S} u_0^{2^*-1}V\\
=-k\displaystyle\int_{S} u_0^{2^*-1}V\\
=O(k\displaystyle\int_S\frac{1}{\lambda^{\frac{N-2}{2}}}\sum_{j=1}^{k}\frac{1}{|y-\xi_j|^{N-2}})\\
= O(k\displaystyle\int_S\frac{1}{\lambda^{\frac{N-2}{2}}}\bigl(
\frac{1}{|y-\xi_1|^{N-2}}+  \frac{k}{|y-\xi_1|^{N-2-\tau}} )
\\
=O\displaystyle\Bigl(\frac{k}{\lambda^{\frac{N-2}{2}}}\bigl( \frac1\lambda+ \frac k{\lambda^{1+\frac\tau2}}\big)\Bigr)=
O\displaystyle\Bigl(\frac{k}{\lambda^{\frac{(1+\delta)(N-2)}{2}}}\Bigr).
\end{array}
\end{equation}

We have that for any $y\in S$

\[
\begin{split}
\sum_{j=2}^{k}\frac{ P_j( \xi_1+\lambda^{-1} y)}{\lambda^{\frac{N-2}2}}\le& C\sum_{j=2}^{k}\frac{ 1}{( 1+ |y - \lambda(\xi_j-\xi_1)|)^{N-2}}\\
\le
&C\sum_{j=2}^{k}\frac{ 1}{| \lambda(\xi_j-\xi_1)|^{N-2}}\\
\le& \frac{ C |\ln k|^{\sigma_N}k^{N-2}}{\lambda^{N-2}}\le
\frac{ C |\ln k|^{\sigma_N}}{\lambda^{\frac{N-2}2}},
\end{split}
\]
where $\sigma_N=0$ if $N\ge 4$ and $\sigma_3=1, $ if $N=3$. So

\[
\begin{split}
&\Bigl|\Bigl(\frac{ V( \xi_1+\lambda^{-1} y)}{\lambda^{\frac{N-2}2}}\Bigr)^{2^*-1}
  -\Bigl(\frac{ P_1( \xi_1+\lambda^{-1} y)}{\lambda^{\frac{N-2}2}}\Bigr)^{2^*-1}\Bigr|\\
  \le & C \Bigl(  \frac1{ (1+|y|)^4   }\frac{  |\ln k|^{\sigma_N} }{\lambda^{\frac{N-2}2}}+ \frac{  |\ln k|^{(2^*-1)\sigma_N}}{\lambda^{\frac{N+2}2}}\Bigr).
\end{split}
\]

Thus, we have

\begin{equation}\label{e8}
\begin{split}
&\int_S V^{2^*-1}u_0\\
=&\frac1{\lambda^{\frac{N-2}2}}\int_{B_{\sqrt\lambda}(0)}\Bigl(\frac{ V( \xi_1+\lambda^{-1} y)}{\lambda^{\frac{N-2}2}}\Bigr)^{2^*-1}u_0( \xi_1+\lambda^{-1} y)
\\
  =&\frac1{\lambda^{\frac{N-2}2}}\int_{B_{\sqrt\lambda}(0)}\Bigl(\frac{ P_1( \xi_1+\lambda^{-1} y)}{\lambda^{\frac{N-2}2}}\Bigr)^{2^*-1}u_0( \xi_1+\lambda^{-1} y)
 \\
   &+\frac1{\lambda^{\frac{N-2}2}}\int_{B_{\sqrt\lambda}(0)}\Bigl[ \Bigl(\frac{ V( \xi_1+\lambda^{-1} y)}{\lambda^{\frac{N-2}2}}\Bigr)^{2^*-1}
  -\Bigl(\frac{ P_1( \xi_1+\lambda^{-1} y)}{\lambda^{\frac{N-2}2}}\Bigr)^{2^*-1}
  \Bigr]u_0( \xi_1+\lambda^{-1} y)
 \\
=& \frac1{\lambda^{\frac{N-2}2}}\int_{B_{\sqrt\lambda}(0)}\Bigl(\frac{ P_1( \xi_1+\lambda^{-1} y)}{\lambda^{\frac{N-2}2}}\Bigr)^{2^*-1}u_0( \xi_1+\lambda^{-1} y)
\\
&+\frac1{\lambda^{\frac{N-2}2}}  O\Bigl( \lambda^{\frac{N-4}2}
 \frac{  |\ln k|^{\sigma_N} }{\lambda^{\frac{N-2}2}}+\lambda^{\frac{N}2} \frac{  |\ln k|^{(2^*-1)\sigma_N}}{\lambda^{\frac{N+2}2}}\Bigr)
 \\
 =& \frac1{\lambda^{\frac{N-2}2}}\int_{B_{\sqrt\lambda}(0)}\Bigl( U(y) +O\bigl( \frac1{\lambda^{N-2}}   \bigr)
\Bigr)^{2^*-1}\Bigl( u_0( \xi_1)+O\bigl(\frac1\lambda\bigr)\Bigr)
\\
&+\frac1{\lambda^{\frac{N-2}2}}  O\Bigl( \lambda^{\frac{N-4}2}
 \frac{  |\ln k|^{\sigma_N} }{\lambda^{\frac{N-2}2}}+\lambda^{\frac{N}2} \frac{  |\ln k|^{(2^*-1)\sigma_N}}{\lambda^{\frac{N+2}2}}\Bigr)\\
 =& \frac{ u_0( \xi_1)}{\lambda^{\frac{N-2}2}}\Bigl( \int_{\mathbb R^N} U^{2^*-1} + O\bigl(\frac1{\lambda^\delta}  \bigr)\Bigr).
\end{split}
\end{equation}

Finally, it is standard to prove

\begin{equation}\label{e9}
\begin{array}{ll}
\quad \displaystyle\frac{1}{2}\int_\Omega|\nabla V|^2-\frac{k}{2^*}\int_S V^{2^*}=k \Bigl(\frac{1}{2}\int_{\Omega_0}|\nabla V|^2-\frac{k}{2^*}\int_S V^{2^*} \Bigr)
\\
=\displaystyle k[\int_{\mathbb{R}^N}|\nabla U_{0, 1}|^2-\frac{1}{2^*}\int_{\mathbb{R}^N}U_{0,1}^{2^*}-\sum_{j=2}^k\frac{B_0}{\lambda^{N-2}|x_j-x_1|}+O(\frac{1}{\lambda^{\frac{N-2}{2}(1+\delta)}})]\\
=k[A-\frac{B_1k^{N-2}}{\lambda^{N-2}r_0^{N-2}}+O(\frac{1}{\lambda^{\frac{N-2}{2}(1+\delta)}})]
\end{array}
\end{equation}

Combining the above obtained results, we get \eqref{new2}.
\\

Finally, the claim follows by the choice of $\lambda$ in \eqref{para}.
 \end{proof}

We are now ready to prove the main theorem.

\begin{proof}[Proof of Theorem~\ref{thm:scaled}: completed]   We apply Proposition~\ref{prop10}.
It is easy to check that
  $F$ has a maximum point at the point $(\ell_0,r_0)$ where $r_0$ maximizes the function $r\to r^{N-2\over 2}u_0(r)$ and  $\ell_0:=\left(\frac {2B }{ C u_0(r_0) r_0^{N-2}}\right)^{\frac2{N-2}}$,
 which is stable under $C^0-$perturbation. Therefore, the reduced energy $I_k$ has a critical point
  $(\ell_k, r_k),$ which produces the solution $U_*+\phi$ to the problem \eqref{main}.

\end{proof}


\appendix
  \renewcommand{\appendixname}{Appendix~\Alph{section}}
\numberwithin{equation}{section}

\section{Non-degeneracy of the positive radial solution}
Without loss of generality we can assume that the annulus is $\A_R:=\{x\in\mathbb R^n\ :\ R\le|x|\le 1\}$
(i.e. $a=R$ and $b=1$).\\

Let $u_R$ be the unique positive radial solution to the following problem:
\begin{equation}\label{1}\left\{\begin{array}{rll}
 -\Delta u &=  u ^p\ &\hbox{in}\ \A_R ,\\
 \ u &=0\ &\hbox{on}\ \partial\A_R .\end{array}\right.
\end{equation}
Here we set
$p:={N+2\over N-2},$ $N\ge3.$

\begin{proposition}\label{prop:non}
There exists a sequence of radii $(R_k)_{k\in\mathbb N},$    such that $u_R$ is non-degenerate for any $R\not = R_k.$
\end{proposition}

\begin{proof}
\begin{itemize}
\item[(i)]
Let us consider the following linear problem:

\begin{equation}\label{2} \left\{\begin{array}{rll}-\Delta v &= p u_R ^{p-1}v\ &\hbox{in}\ \A_R ,\\
 v &=0\ &\hbox{on}\ \partial\A _R .\end{array}\right.
\end{equation}

We denote by $\lambda_k= k(k+n-2)$ for $k=0,1,2,...$ the eigenvalues of $-\Delta$ on the sphere $\mathbb S^{n-1}.$
Let $\{\Phi_i^k\ :\ 1\le i\le m_k\}$ denote a basis for the $k^{th}$ eigenspace of $-\Delta.$
Then for any function $v=v(r,\theta)$ on the annulus $\mathcal A_R$ we may write
\begin{equation}\label{21}v(r,\theta)=\sum_{k\ge0 } a_{ k}(r)\tilde \Phi_i^k(\theta),\quad r\in(1,R),\ \theta\in\mathbb S^{n-1},\end{equation}
   where each $a_k$ is a radial solution to
   \begin{equation}\label{3}\left\{\begin{array}{ll} a_k^{\prime\prime}+{n-1\over r}a_k^{\prime}+\left(p u_R^{p-1} (r)-{\lambda_k\over r^2}\right)a_k(r) = 0\ \hbox{in}\ (R,1) ,\\
    a_k(R)=a_k(1) =0,\end{array}\right.
\end{equation}
and
   $$\tilde \Phi_i^k(\theta)=\sum\limits_{i=1}^{m_k} c_i\Phi_i^k(\theta), \quad \hbox{for some}\ c_i\in\mathbb R.$$

  \item[(ii)] Argue as in Proposition 2.1 in \cite{bcgp}, we have that
   \begin{equation}\label{4}
   a_0(r) =0\ \hbox{for any}\ r\in(R,1).
   \end{equation}
   It means that $u_0$ is non-degenerate in the space of radial functions.

\item[(iii)] For any integer $k\ge1,$ let $ {\mu_ k}_i={\mu_ k}_i(R),$ ${i\ge1}$ be the sequence of the eigenvalues
of the problem:
\begin{equation}\label{5}\left\{\begin{array}{ll}\phi^{\prime\prime}+{n-1\over r}\phi^{\prime}+\left(p u_R^{p-1} (r)-{\lambda_k\over r^2}\right)\phi=-{\mu_ k}_i\phi\ \hbox{in}\ (R,1) ,\\
 \phi(R)=\phi(1) =0.\end{array}\right.
\end{equation}
We point out that if
\begin{equation}\label{6}
{\mu_ k}_i(R)\not=0\ \hbox{for any $k\ge1$ and $i\ge1$},
\end{equation}
then any solutions to \eqref{3} $a_k\equiv0.$

So by \eqref{21} (together with \eqref{4}) we deduce that any solutions to \eqref{2} $v\equiv0$, i.e. $u_R$ is non-degenerate.

\item[(iv)] By Corollary  2.4  in \cite{l} we get
\begin{equation}\label{7}{\mu_ 1}_1 (R)<0\ \hbox{and}\
{\mu_ k}_i (R)>0 \ \hbox{for  $k\ge1$ and $i\ge2$},\quad \hbox{for any $R\in(0,1)$}.
\end{equation}
It only remains to check the behavior of the first eigenvalue ${\mu_k}_1(R)$ for any $k\ge2.$
We know  by Lemma 3.1 in   \cite{l} that
$$\lim\limits_{R\to1}{\mu_k}_1(R)=-\infty,\ \hbox{for any}\ k\ge1.$$

\item[(v)]
If $\phi$ solves \eqref{5} then
$\psi(t)=\phi(t(1-R)+2R-1)$ solves the following problem:
\begin{equation}\label{51}
\left\{\begin{array}{ll}
\psi^{\prime\prime}+{(n-1)(1-R)\over t(1-R)+2R-1}\psi^{\prime}+(1-R)^2\left(p w_R^{p-1} (t)-{\lambda_k\over \(t(1-R)+2R-1\)^2}\right)\psi=\lambda_k\psi,\ \hbox{in}\ (1,2), \\
\psi(1)=\psi(2) =0.
\end{array}\right.
\end{equation}
Where
$$\lambda_k=\lambda_k(R):=-(1-R)^2{\mu_k}_1(R).$$
On the other hand, we see that  $w_R(t)=u_R(t(1-R)+2R-1)$ solves the following problem:
\begin{equation}\label{52} \left\{\begin{array}{ll} w_R^{\prime\prime}+{(n-1)(1-R)\over t(1-R)+2R-1}w_R^{\prime}+(1-R)^2 w_R^{p } (t)=0\ \hbox{in}\ (1,2) ,\\
 w_R(1)=w_R(2) =0  .
\end{array}\right.\end{equation}

\item[(vi)] We claim that

{\em
for any $k\ge2$ there exists a finite number of radii ${R_k}_1,\dots,{R_k}_{\ell(k)}$ such that
$\lambda_k({R_k}_i)=0$ for $i=1,\dots,{\ell(k)}.$}

The proof for the claim  could follow  the same arguments as in Lemma 2.2 (c) of \cite{dw}.
Indeed,
 using a result due to Kato (see Example 2.12, page 380 in \cite{k}),   we could prove that each function $R\to\lambda_k(R)$ is analytic so it can only vanish at a finite number of points. We can prove that the function $W:(0,1)\to C^2(I)$, $I=[1,2],$ defined by
$W(R)(t)=w_R(t)  $ is analytic using the same arguments developed by  Dancer in\cite{d}.

     \end{itemize}
\end{proof}

\bibliographystyle{springer}
\newcommand{\noopsort}[1]{} \newcommand{\printfirst}[2]{#1}
\newcommand{\singleletter}[1]{#1} \newcommand{\switchargs}[2]{#2#1}

\end{document}